\documentclass[11pt]{amsart}
\usepackage[utf8]{inputenc}
\usepackage[english]{babel}
\usepackage{multirow}
\usepackage{anysize}
\usepackage{color}
\usepackage{multicol}

\usepackage{enumitem}
\usepackage{amsfonts}
\usepackage{amsmath}
\usepackage{amssymb}
\usepackage{amsthm}

\usepackage[usenames,dvipsnames]{pstricks}
 \usepackage{pstricks-add}
 \usepackage{epsfig}

\newtheorem{teo}{Theorem}[section]

\newtheorem{prop}[teo]{Proposition}

\newtheorem{lem}[teo]{Lemma}
\newtheorem{defn}[teo]{Definition}
\newtheorem{ob}[teo]{Remark}

\title{\bf $|3|-$gradings of complex simple Lie algebras}

\author{Mauricio Godoy Molina and Diego Lagos}

\thanks{This research is partially supported by Fondecyt \#1181084}

\begin{document}
\maketitle
\begin{abstract}
The aim of this paper is to investigate the algebraic structure that appears on $|3|-$gradings $\mathfrak{n}=\mathfrak{n}_{-3}\oplus \cdots \oplus \mathfrak{n}_3$ of a complex simple Lie algebra $\mathfrak{n}$. In particular, we completely determine the possible reductive algebras $\mathfrak{n}_0$ and prove that the only free nilpotent Lie algebra of step 3 that appears as the negative part  $\mathfrak{n}_{-3}\oplus\mathfrak{n}_{-2}\oplus\mathfrak{n}_{-1}$ of a grading is the usual $|3|-$grading of the exceptional Lie algebra $\mathfrak{g}_2$.
\end{abstract}

\section{Introduction}

A differential system is a pair $(M,D)$, where $M$ is a differentiable manifold and $D\hookrightarrow TM$ is a distribution on $M$, that is, a vector sub-bundle of the tangent bundle of $M$. These objects appear naturally when studying constrained mechanics, where $M$ is the configuration space of a mechanical system and $D$ encodes the admissible velocities. There is a vast amount of literature regarding these mathematical objects, since they play an important role in contact geometry \cite{klr}, sub-Riemannian geometry \cite{abb,m} and geometric control theory \cite{j}.

The study of symmetries of differential systems has been a relevant tool in differential geometry for over a century. For example, the seminal paper by \'E. Cartan \cite{c} is nowadays understood as a criterion of the flatness of differential systems with $M$ a five dimensional manifold and $D$ of rank two, by using the symmetries of such a distribution. For the sake of context, recall that the group of global symmetries of a differential system $(M,D)$ is 
\[
{\rm Sym}(M,D)=\{\varphi\colon M\to M\mbox{ diffeomorphism }|\;\varphi_*D=D\}.
\]
This group is extremely difficult to determine completely, except for a few special cases. A well-known example is the Legendre transform in ${\mathbb R}^{2n+1}$, which is a global symmetry for the canonical contact structure $D_{cont}$ (for details, see \cite{klr}). In the more general situation of a contact manifold, it is possible to find criteria for the simplicity of the group of $C^r$ contactomorphisms, see \cite{tsu}.

As usual in differential geometry, the infinitesimal object is easier to deal with, namely the Lie algebra of infinitesimal symmetries of $(M,D)$, given by
\[
{\rm sym}(M,D)=\{X\in{\mathfrak X}(M)\;|\;[X,\Gamma(D)]\subseteq\Gamma(D)\},
\]
where $\Gamma(D)$ denotes the Lie algebra of sections of the distribution $D$. In this context, the infinitesimal symmetries of important differential systems can be found explicitly, for example, the Lie algebra of infinitesimal contactomorphisms ${\rm sym}({\mathbb R}^{2n+1},D_{cont})$ is the infinite dimensional jet space $J({\mathbb R}^{2n+1})$, see \cite{klr}.

The search for a way to determine the infinitesimal symmetries of special differential systems has proved fruitful over the years, especially as a consequence of the fundamental work by Tanaka \cite{tanaka}, where an explicit linear algebraic procedure is given to determine ${\rm sym}(N,{\mathfrak n}_{-1})$ in the case where $N$ is the (unique, up to isomorphism, connected and simply connected) nilpotent Lie group associated to a graded nilpotent Lie algebra ${\mathfrak n}={\mathfrak n}_{-k}\oplus\cdots\oplus{\mathfrak n}_{-1}$. This process is referred to as {\it Tanaka prolongation}. 

Using techniques from parabolic geometry, see \cite{cap}, the study of these very particular spaces of infinitesimal symmetries is related to $|k|-$gradings of semi-simple Lie algebras. The aim of this paper is to provide some details in the case of $k=3$ for the classical Lie algebras. From a different point of view, we can ask whether the nilpotent part of a $|3|-$grading is a Lie algebra of a certain kind (an idea successfully exploited in \cite{fu} for the case of $|2|-$gradings). In addition, we present a more concrete alternative to the general result obtained in \cite{ben} concerning the free nilpotent Lie algebras of step 3.

The structure of the paper is the following. In Section \ref{gradings} we provide some basic ideas of the theory of gradings of simple Lie algebras and deduce the possible dimensions of the graded subspaces. In Section \ref{n0} we present the structure of the reductive subalgebra $\mathfrak{n}_0$ found in the grading. In Section \ref{freecl} we focus on the case of free Lie algebras and try to answer the question of whether they appear as negative part of $|3|-$gradings of simple Lie algebras. 


\section{Simple Lie algebras and $|3|-$gradings}\label{gradings}


In this section we describe all possible dimensions of each degree in the $|3|-$gradings associated to complex simple Lie algebras.

The following theorem is one interpretation of Theorem 3.2.1 in \cite{cap}.

\begin{teo}\label{th:capslovak}
Let $\mathfrak{n}$ be a simple Lie algebra with Cartan subalgebra $\mathfrak{h}$ and set of simple roots  $\Delta^0.$ Then, the $|k|-$gradings of $\mathfrak{n}$ are in bijection with the subsets $\Sigma\subseteq \Delta^0$ such that the height with respect to $\Sigma$ of the highest root $\theta$ of $\mathfrak{n}$ equals $k$. 
\end{teo}

From the general theory of root systems, see \cite{hum}, every root can be written as
\[
\alpha=a_1\alpha_1+\cdots+a_n\alpha_n,
\]
where $\Delta^0=\{\alpha_1,\hdots,\alpha_n\}$ and $a_1,\hdots,a_n$ are integers all of the same sign or zero. Recall that the height with respect to $\Sigma\subseteq\Delta^0$ of a root $\alpha$ of a simple Lie algebra $\mathfrak{n}$, denoted by $ht_{\Sigma}(\alpha)$, is defined as
\[
ht_{\Sigma}(\alpha)=\sum_{\alpha_i\in\Sigma}a_i.
\]
The highest weight roots for the complex simple Lie algebras are well-known and can be easily found in many classical books, for example \cite{bourbaki}. For the ease of the reader, we summarize these roots in Table \ref{highroots}.

\begin{table}[ht]
\caption{Highest roots of simple Lie algebra over $\mathbb{C}$}\label{highroots}
\begin{center}
	\begin{tabular}{|l|l|lll}
		\cline{1-2}
		Lie algebra	& Highest root \\ \hline\hline
		$A_n$, $n\geq 1$  & $\alpha_1+\alpha_2+\cdots +\alpha_n$  \\ \cline{1-2}
		$B_n$, $n\geq 2$ &  $\alpha_1+2\alpha_2+2\alpha_3+\cdots+2\alpha_n$\\ \cline{1-2}
		$C_n$, $n\geq 3$ &  $2\alpha_1+2\alpha_2+2\alpha_3+\cdots+\alpha_n$\\ \cline{1-2}
		$D_n$, $n\geq 4$ &  $\alpha_1+2\alpha_2+2\alpha_3+\cdots+2\alpha_{n-2}+\alpha_{n-1}+\alpha_{n}$ \\ \cline{1-2}
		$\mathfrak{e}_6$ & $\alpha_1+2\alpha_2+2\alpha_3+3\alpha_4+2\alpha_5+\alpha_6$ \\ \cline{1-2}
		$\mathfrak{e}_7$ & $2\alpha_1+2\alpha_2+3\alpha_3+4\alpha_4+3\alpha_5+2\alpha_6+\alpha_7$ \\ \cline{1-2}
		$\mathfrak{e}_8$ & $2\alpha_1+3\alpha_2+4\alpha_3+6\alpha_4+5\alpha_5+4\alpha_6+3\alpha_7+2\alpha_8$ \\
		\cline{1-2}
		$\mathfrak{f}_4$ & $2\alpha_1+3\alpha_2+4\alpha_3+2\alpha_4$ \\ \cline{1-2}
		$\mathfrak{g}_2$ & $3\alpha_1+2\alpha_2$ \\ \cline{1-2}
	\end{tabular}
\end{center}
\end{table}

The goal of this section is to compute the dimensions of all the graded subspaces of a given a $|3|$-grading $\mathfrak{n}_{-3}\oplus \cdots \oplus \mathfrak{n}_3$ of a simple Lie algebra $\mathfrak{n}$.

\subsection{The case $A_{n}$, $n\geq 4$}
Let $\mathfrak{n}$ be a Lie algebra of type $A_{n}.$ Since the highest root of $\mathfrak{n}$ is $\alpha_1+\cdots +\alpha_n$, then Theorem \ref{th:capslovak} implies that the $|3|-$gradings of $\mathfrak{n}$ are in correspondence with the subsets of simple roots 
\[
\Sigma_{i,j,k}=\{\alpha_i, \alpha_j,\alpha_k\},\quad\mbox{where $1\leq i<j<k\leq n$.}
\]
Using the standard grading of $A_n$ given in \cite{yam}, we can represent the $|3|-$grading $\mathfrak{n}_{-3}\oplus \cdots \oplus \mathfrak{n}_3$ as follows:

\newpage 

{\small
\[
\begin{array}{cccccccccccc}
&&&i&&&&j&&&&k
\end{array}
\]
\[
\begin{array}{c}
i\\
\\
\\
j\\
\\
\\
k
\\
\end{array}
\left(
\begin{array}{ccc|ccc|ccc|ccc}
&&&&&&&&&&&\\
&{\mathfrak{n}}_{0}&&&{\mathfrak{n}}_{1}&&&{\mathfrak{n}}_{2}&&&{\mathfrak{n}}_{3}&\\
&&&&&&&&&&&\\\hline
&&&&&&&&&&&\\
&{\mathfrak{n}}_{-1}&&&{\mathfrak{n}}_{0}&&&{\mathfrak{n}}_{1}&&&{\mathfrak{n}}_{2}&\\
&&&&&&&&&&&\\\hline
&&&&&&&&&&&\\
&{\mathfrak{n}}_{-2}&&&{\mathfrak{n}}_{-1}&&&{\mathfrak{n}}_{0}&&&{\mathfrak{n}}_{1}&\\
&&&&&&&&&&&\\\hline
&&&&&&&&&&&\\
&{\mathfrak{n}}_{-3}&&&{\mathfrak{n}}_{-2}&&&{\mathfrak{n}}_{-1}&&&{\mathfrak{n}}_{0}&\\
&&&&&&&&&&&
\end{array}
\right)
\]}

The dimensions of the negatively graded subspaces in the grading above follow from a direct computation and can be summarized in the following proposition.

\begin{prop}
	Let $\mathfrak{n}$ be  a Lie algebra of type $A_{n}.$ If $\mathfrak{n}_{-3}\oplus \cdots \oplus \mathfrak{n}_3$ is the $|3|-$grading of $\mathfrak{n}$ corresponding to $\Sigma_{i,j,k}$, where $1\leq i<j<k\leq n$, then 
	\begin{align*}
		\dim \mathfrak{n}_{-1}&=i(j-i)+(k-j)(j-i)+(n+1-k)(k-j),\\
	    \dim \mathfrak{n}_{-2}&=i(k-j)+(n+1-k)(j-i),\\
	    \dim \mathfrak{n}_{-3}&=i(n+1-k).
	\end{align*}      
\end{prop}

\subsection{ The case $B_n$.}  Let $\mathfrak{n}$ be a Lie algebra of type $B_{n}.$ Since the highest root of $\mathfrak{n}$ is $\alpha_{1}+2\alpha_2+\cdots+2\alpha_n$ then Theorem \ref{th:capslovak} implies that the $|3|-$gradings of $\mathfrak{n}$ are in
correspondence with the subsets of simple roots
\[
\Sigma_{i}=\{\alpha_1,\alpha_i\}\subseteq \Delta^0,\quad\mbox{where $2\leq i \leq n$}.
\]
The grading $\mathfrak{n}_{-3}\oplus \cdots \oplus \mathfrak{n}_{3}$ given in \cite{yam} for $B_n$ can represented as follows:

{\small
\[\hspace{1cm}
1
\hspace{1.5cm}
i
\hspace{1.5cm}
i
\hspace{1.5cm}
1
\]
\[
\begin{array}{c}
\\
1\\
\\
\\
i\\
\\
\\
i\\
\\
\\
1
\\
\\
\end{array}
\left(
\begin{array}{ccc|ccc|ccc|ccc|ccc}
&&&&&&&&&&&&\\
&{\mathfrak{n}}_{0}&&&{\mathfrak{n}}_{1}&&&{\mathfrak{n}}_{2}&&&{\mathfrak{n}}_{3}&&& \ast&\\
&&&&&&&&&&&&\\\hline
&&&&&&&&&&&&\\
&{\mathfrak{n}}_{-1}&&&{\mathfrak{n}}_{0}&&&{\mathfrak{n}}_{1}&&&{\mathfrak{n}}_{2}&&& {\mathfrak{n}}_{3}&\\
&&&&&&&&&&&&\\\hline
&&&&&&&&&&&&\\
&{\mathfrak{n}}_{-2}&&&{\mathfrak{n}}_{-1}&&&{\mathfrak{n}}_{0}&&&{\mathfrak{n}}_{1}&&& \mathfrak{n}_{2}&\\
&&&&&&&&&&&&\\\hline
&&&&&&&&&&&&\\
&{\mathfrak{n}}_{-3}&&&{\mathfrak{n}}_{-2}&&&{\mathfrak{n}}_{-1}&&&{\mathfrak{n}}_{0}&&& \mathfrak{n}_{1}&\\
&&&&&&&&&&&&\\\hline
&&&&&&&&&&&&\\
& \ast &&&{\mathfrak{n}}_{-3}&&&{\mathfrak{n}}_{-2}&&&{\mathfrak{n}}_{-1}&&& \mathfrak{n}_{0}&\\
&&&&&&&&&&&&
\end{array}
\right)
\]}

As before, the dimensions of the negatively graded subspaces in the grading for $B_n$ follow
from a simple combinatorial argument and can be summarized in the following proposition.

\begin{prop}\label{prop:dimBn}
	Let be $\mathfrak{n}$ a Lie algebra of type $B_{n}.$ If $\mathfrak{n}_{-3}\oplus \cdots \oplus \mathfrak{n}_3$ is the $|3|-$grading of $\mathfrak{n}$ associated to $\Sigma_{i}$, where $2\leq i \leq n$, then 
\begin{align*}
	\dim \mathfrak{n}_{-1} &= (i-1)+(i-1)(2n+1-2i)=(i-1)(2n+2-2i),\\ 
	\dim \mathfrak{n}_{-2} &= \frac{(i-1)(i-2)}{2}+(i-1)(2n-2i),\\ 
	\dim \mathfrak{n}_{-3} &= (i-1).
\end{align*}

\end{prop}

\subsection{The case $C_n$.}  Let $\mathfrak{n}$ be a Lie algebra of type $C_{n}.$ The highest weight root of $\mathfrak{n}$ is $2\alpha_{1}+2\alpha_2+\cdots+2\alpha_{n-1}+\alpha_n$ and so the $|3|-$gradings $\mathfrak{n}_{-3}\oplus \cdots \oplus \mathfrak{n}_{3}$ of $\mathfrak{n}$ are in one-to-one correspondence with the subsets 
\[
\Sigma_{i}=\{\alpha_i,\alpha_n\}\subseteq \Delta^0,\quad\mbox{where $1\leq i \leq n-1$.}
\]
Using the standard grading given in \cite{yam} we can represent the $|3|-$grading $\mathfrak{n}_{-3}\oplus \cdots \oplus \mathfrak{n}_3$ as follows:

{\small
\[
\begin{array}{cccccccccccc}
&&&i&&&&n&&&&i
\end{array}
\]
\[
\begin{array}{c}
i\\
\\
\\
n\\
\\
\\
i
\\
\end{array}
\left(
\begin{array}{ccc|ccc|ccc|ccc}
&&&&&&&&&&&\\
&{\mathfrak{n}}_{0}&&&{\mathfrak{n}}_{1}&&&{\mathfrak{n}}_{2}&&&{\mathfrak{n}}_{3}&\\
&&&&&&&&&&&\\\hline
&&&&&&&&&&&\\
&{\mathfrak{n}}_{-1}&&&{\mathfrak{n}}_{0}&&&{\mathfrak{n}}_{1}&&&{\mathfrak{n}}_{2}&\\
&&&&&&&&&&&\\\hline
&&&&&&&&&&&\\
&{\mathfrak{n}}_{-2}&&&{\mathfrak{n}}_{-1}&&&{\mathfrak{n}}_{0}&&&{\mathfrak{n}}_{1}&\\
&&&&&&&&&&&\\\hline
&&&&&&&&&&&\\
&{\mathfrak{n}}_{-3}&&&{\mathfrak{n}}_{-2}&&&{\mathfrak{n}}_{-1}&&&{\mathfrak{n}}_{0}&\\
&&&&&&&&&&&
\end{array}
\right)
\]}

As in the two previous cases, we can conclude that the dimensions are given in a similar manner.

\begin{prop}\label{prop:dimCn}
Let $\mathfrak{n}$ be a Lie algebra of type $C_n.$ If $\mathfrak{n}_{-3}\oplus \cdots \oplus \mathfrak{n}_3$ is a $|3|-$grading   of  $\mathfrak{n}$ corresponding to $\Sigma_{i}=\{\alpha_i,\alpha_n\}\subseteq \Delta^0$, where $1\leq i \leq n-1$, then 
	\begin{align*}
		\dim \mathfrak{n}_{-1} &= i(n-i)+(n-i)+\frac{(n-i)(n-i-1)}{2}\\ 
		\dim \mathfrak{n}_{-2} &= (n-i)^2\\ 
		\dim \mathfrak{n}_{-3} &= i+\frac{i(i-1)}{2}
	\end{align*}
	
\end{prop}

\subsection{The case $D_n$.} Let $\mathfrak{n}$ be a Lie algebra of type $D_{n}.$ The highest weight root of $\mathfrak{n}$ is $\alpha_1+2\alpha_2+2\alpha_3+\cdots+2\alpha_{n-2}+\alpha_{n-1}+\alpha_{n}$ and so the $|3|-$gradings of $\mathfrak{n}$ are in one-to-one correspondence with the the following subsets of simple roots of $\mathfrak{n}$:
\begin{multicols}{2}
\begin{enumerate}
\item  $\Sigma_{1,i}=\{\alpha_1,\alpha_{i}\},$
\item  $\Sigma_{i,n}=\{\alpha_i,\alpha_{n}\},$
\item $\Sigma_{i,n-1}=\{\alpha_i,\alpha_{n-1}\},$
\item  $\Sigma_{1,n-1,n}=\{\alpha_1,\alpha_{n-1},\alpha_n\},$
\end{enumerate}
\end{multicols}
\noindent where $2\leq i\leq n-2.$

There exists an automorphism of the Dynkin diagram that permutes $\alpha_{n-1}$ and $\alpha_n$, and thus, there exists an automorphism of $D_n$ giving an isomorphism between the gradings induced by $\Sigma_{i,n-1}$ and $\Sigma_{i,n}$, see \cite[Chapter 14]{hum}. Therefore, in the following proposition we only consider the sets $\Sigma_{i,1}=\{\alpha_1,\alpha_{i}\}$, $\Sigma_{i,n}=\{\alpha_i,\alpha_{n}\}$ and $\Sigma_{1,n-1,n}=\{\alpha_1,\alpha_{n-1},\alpha_n\}.$ Using the standard grading given in \cite{yam} we have:

\newpage

\begin{itemize}

\item For $\Sigma_{i,1}$ the $|3|-$grading is given by
{\small
\[\hspace{1cm}
1
\hspace{1.5cm}
i
\hspace{1.5cm}
i
\hspace{1.5cm}
1
\]
\[
\begin{array}{c}
\\
1\\
\\
\\
i\\
\\
\\
i\\
\\
\\
1
\\
\\
\end{array}
\left(
\begin{array}{ccc|ccc|ccc|ccc|ccc}
&&&&&&&&&&&&\\
&{\mathfrak{n}}_{0}&&&{\mathfrak{n}}_{1}&&&{\mathfrak{n}}_{2}&&&{\mathfrak{n}}_{3}&&& \ast&\\
&&&&&&&&&&&&\\\hline
&&&&&&&&&&&&\\
&{\mathfrak{n}}_{-1}&&&{\mathfrak{n}}_{0}&&&{\mathfrak{n}}_{1}&&&{\mathfrak{n}}_{2}&&& {\mathfrak{n}}_{3}&\\
&&&&&&&&&&&&\\\hline
&&&&&&&&&&&&\\
&{\mathfrak{n}}_{-2}&&&{\mathfrak{n}}_{-1}&&&{\mathfrak{n}}_{0}&&&{\mathfrak{n}}_{1}&&& \mathfrak{n}_{2}&\\
&&&&&&&&&&&&\\\hline
&&&&&&&&&&&&\\
&{\mathfrak{n}}_{-3}&&&{\mathfrak{n}}_{-2}&&&{\mathfrak{n}}_{-1}&&&{\mathfrak{n}}_{0}&&& \mathfrak{n}_{1}&\\
&&&&&&&&&&&&\\\hline
&&&&&&&&&&&&\\
& \ast &&&{\mathfrak{n}}_{-3}&&&{\mathfrak{n}}_{-2}&&&{\mathfrak{n}}_{-1}&&& \mathfrak{n}_{0}&\\
&&&&&&&&&&&&
\end{array}
\right)
\]}

\item For $\Sigma_{i,n}$ the $|3|-$grading is given by
{\small
\[
\begin{array}{cccccccccccc}
&&&i&&&&n&&&&i
\end{array}
\]
\[
\begin{array}{c}
i\\
\\
\\
n\\
\\
\\
i
\\
\end{array}
\left(
\begin{array}{ccc|ccc|ccc|ccc}
&&&&&&&&&&&\\
&{\mathfrak{n}}_{0}&&&{\mathfrak{n}}_{1}&&&{\mathfrak{n}}_{2}&&&{\mathfrak{n}}_{3}&\\
&&&&&&&&&&&\\\hline
&&&&&&&&&&&\\
&{\mathfrak{n}}_{-1}&&&{\mathfrak{n}}_{0}&&&{\mathfrak{n}}_{1}&&&{\mathfrak{n}}_{2}&\\
&&&&&&&&&&&\\\hline
&&&&&&&&&&&\\
&{\mathfrak{n}}_{-2}&&&{\mathfrak{n}}_{-1}&&&{\mathfrak{n}}_{0}&&&{\mathfrak{n}}_{1}&\\
&&&&&&&&&&&\\\hline
&&&&&&&&&&&\\
&{\mathfrak{n}}_{-3}&&&{\mathfrak{n}}_{-2}&&&{\mathfrak{n}}_{-1}&&&{\mathfrak{n}}_{0}&\\
&&&&&&&&&&&
\end{array}
\right)
\]}

\item For $\Sigma_{1,n-1,n}$ the $|3|-$grading is given by

{\small
\[
\begin{array}{cccccccccccc}
&&&n-1&&&&n&&&&n-1
\end{array}
\]
\[
\begin{array}{c}
n-1\\
\\
\\
n\\
\\
\\
n-1
\\
\end{array}
\left(
\begin{array}{ccc|ccc|ccc|ccc}
&&&&&&&&&&&\\
&{\mathfrak{n}}_{0}&&&{\mathfrak{n}}_{1}&&&{\mathfrak{n}}_{2}&&&{\mathfrak{n}}_{3}&\\
&&&&&&&&&&&\\\hline
&&&&&&&&&&&\\
&{\mathfrak{n}}_{-1}&&&{\mathfrak{n}}_{0}&&&{\mathfrak{n}}_{1}&&&{\mathfrak{n}}_{2}&\\
&&&&&&&&&&&\\\hline
&&&&&&&&&&&\\
&{\mathfrak{n}}_{-2}&&&{\mathfrak{n}}_{-1}&&&{\mathfrak{n}}_{0}&&&{\mathfrak{n}}_{1}&\\
&&&&&&&&&&&\\\hline
&&&&&&&&&&&\\
&{\mathfrak{n}}_{-3}&&&{\mathfrak{n}}_{-2}&&&{\mathfrak{n}}_{-1}&&&{\mathfrak{n}}_{0}&\\
&&&&&&&&&&&
\end{array}
\right)
\]}
\end{itemize}

As a consequence, for $|3|-$gradings of Lie algebras of type $D_n$, one has to take into account the three situations described above separately. The dimensions of the negatively graded subspaces are contained in the following proposition.

\newpage

\begin{prop}\label{prop:dimDn}
Let $\mathfrak{n}$ be a Lie algebra of type $D_n$ and $\mathfrak{n}_{-3}\oplus \cdots \oplus \mathfrak{n}_3$ a $|3|-$grading of $\mathfrak{n}$ associated to a subset $\Sigma$ of simple roots of $\mathfrak{n}.$ Then

\begin{enumerate}
\item For $\Sigma=\Sigma_{i,1}$, where $2\leq i\leq n-2$, we have
	\begin{align*}
		\dim \mathfrak{n}_{-1} &= (2n-2i)(i-1)+i-1,\\ 
		\dim \mathfrak{n}_{-2} &= \frac{(i-1)(i-2)}{2}+(2n-2i),\\ 
		\dim \mathfrak{n}_{-3} &= i-1.
	\end{align*}
\item For $\Sigma=\Sigma_{i,n}$, where $2\leq i\leq n-2$, we have	
	\begin{align*}
		\dim \mathfrak{n}_{-1} &= \frac{(n-i)(n-i-1)}{2}+i(n-i),\\ 
		\dim \mathfrak{n}_{-2} &= i(n-i),\\ 
		\dim \mathfrak{n}_{-3} &= \frac{i(i-1)}{2}.
	\end{align*}
	
\item For $\Sigma=\Sigma_{1,n-1,n}$ we have
	\begin{align*}
		\dim \mathfrak{n}_{-1} &= 3(n-2),\\ 
		\dim \mathfrak{n}_{-2} &=2+\frac{(n-2)(n-3)}{2},\\ 
		\dim \mathfrak{n}_{-3} &= n-2.
	\end{align*}	
\end{enumerate} 
\end{prop}

\subsection{The cases of exceptional Lie algebras}
If $\mathfrak{n}$ is one of the five exceptional Lie algebras $\mathfrak{e}_6$, $\mathfrak{e}_7$, $\mathfrak{e}_8$, $\mathfrak{f}_4$ or $\mathfrak{g}_2$, it is also possible to compute the dimensions of the negatively graded subspaces for a $|3|-$grading associated to appropriate subsets $\Sigma$ of simple roots. The idea is to use the general theory of gradings, that can be found in \cite{cap}.

Recall that the Killing form of $\mathfrak{n}$ restricts to a duality pairing between $\mathfrak{n}_i$ and $\mathfrak{n}_{-i}$, and thus, the dimensions of $\mathfrak{n}_i$ and $\mathfrak{n}_{-i}$ are equal. Moreover, these spaces are built from the set of positive roots explicitly as $$\mathfrak{n}_{i}=\displaystyle{\bigoplus_{ht_{\Sigma}(\alpha)=i}\mathfrak{g}_{\alpha}} \quad (i=1,2,3)$$
where the sum is over all positive roots $\alpha$ with $ht_{\Sigma}(\alpha)=i$.

The following propositions follow from lengthy computations employing the full sets of simple roots for each exceptional Lie algebra. The explicit expressions that were used are the ones that can be found in \cite{bourbaki}. In each case, we indicate the subsets of simple roots that produce non isomorphic $|3|-$gradings.

\begin{prop}[The case of $\mathfrak{g}_2$]\label{propg2}
There is a unique $|3|-$grading $\mathfrak{n}_{-3}\oplus \cdots \oplus \mathfrak{n}_{3}$ of $\mathfrak{g}_2$, which is associated to $\Sigma=\{\alpha_1\}$. The dimensions of the negatively graded subspaces are
\[
\dim\mathfrak{n}_{-3}=2,\quad\dim\mathfrak{n}_{-2}=1,\quad\dim\mathfrak{n}_{-1}=2.
\]
\end{prop}

Notice that the $|3|-$grading in Proposition \ref{propg2} corresponds to the well-known Cartan grading of $\mathfrak{g}_2$, see \cite{c}.

\begin{prop}[The case of $\mathfrak{f}_4$]
There is a unique $|3|-$grading $\mathfrak{n}_{-3}\oplus \cdots \oplus \mathfrak{n}_{3}$ of $\mathfrak{f}_4$, which is associated to $\Sigma=\{\alpha_2\}$. The dimensions of the negatively graded subspaces are
\[
\dim\mathfrak{n}_{-3}=2,\quad\dim\mathfrak{n}_{-2}=6,\quad\dim\mathfrak{n}_{-1}=12.
\]
\end{prop}

\newpage

\begin{prop}[The case of $\mathfrak{e}_6$] There are four non isomorphic $|3|-$gradings $\mathfrak{n}_{-3}\oplus \cdots \oplus \mathfrak{n}_{3}$ of $\mathfrak{e}_6$. These are summarized in Table \ref{e6}.

\begin{table}[ht]
\caption{$|3|-$gradings of the exceptional Lie algebra $\mathfrak{e}_6$}\label{e6}
\begin{center}
    \begin{tabular}{|l|l|lll}
		\cline{1-2}
		$\Sigma$	& Dimensions \\ \hline\hline
		$\Sigma=\{\alpha_1,\alpha_2\}$ & $\dim\mathfrak{n}_{-3}=1$, $\dim\mathfrak{n}_{-2}=10$, $\dim\mathfrak{n}_{-1}=15$  \\ 
  \hline
  $\Sigma=\{\alpha_1,\alpha_3\}$ & $\dim\mathfrak{n}_{-3}=5$, $\dim\mathfrak{n}_{-2}=10$, $\dim\mathfrak{n}_{-1}=11$  \\
  \hline
  $\Sigma=\{\alpha_1,\alpha_5\}$ & $\dim\mathfrak{n}_{-3}=4$, $\dim\mathfrak{n}_{-2}=9$, $\dim\mathfrak{n}_{-1}=16$  \\
  \hline
  $\Sigma=\{\alpha_4\}$ & $\dim\mathfrak{n}_{-3}=2$, $\dim\mathfrak{n}_{-2}=10$, $\dim\mathfrak{n}_{-1}=17$  \\\hline
	\end{tabular}
\end{center}
\end{table}
\end{prop}

\begin{prop}[The case of $\mathfrak{e}_7$]
There are five non isomorphic $|3|-$gradings $\mathfrak{n}_{-3}\oplus \cdots \oplus \mathfrak{n}_{3}$ of $\mathfrak{e}_7$. These are summarized in Table \ref{e7}.

\begin{table}[ht]
\caption{$|3|-$gradings of the exceptional Lie algebra $\mathfrak{e}_7$}\label{e7}
\begin{center}
    \begin{tabular}{|l|l|lll}
		\cline{1-2}
		$\Sigma$	& Dimensions \\ \hline\hline
		$\Sigma=\{\alpha_1,\alpha_7\}$ & $\dim\mathfrak{n}_{-3}=1$, $\dim\mathfrak{n}_{-2}=17$, $\dim\mathfrak{n}_{-1}=25$  \\ 
  \hline
  $\Sigma=\{\alpha_2,\alpha_7\}$ & $\dim\mathfrak{n}_{-3}=6$, $\dim\mathfrak{n}_{-2}=17$, $\dim\mathfrak{n}_{-1}=25$  \\
  \hline
  $\Sigma=\{\alpha_6,\alpha_7\}$ & $\dim\mathfrak{n}_{-3}=10$, $\dim\mathfrak{n}_{-2}=17$, $\dim\mathfrak{n}_{-1}=16$  \\
  \hline
  $\Sigma=\{\alpha_3\}$ & $\dim\mathfrak{n}_{-3}=2$, $\dim\mathfrak{n}_{-2}=15$, $\dim\mathfrak{n}_{-1}=30$  \\
  \hline
  $\Sigma=\{\alpha_5\}$ & $\dim\mathfrak{n}_{-3}=5$, $\dim\mathfrak{n}_{-2}=15$, $\dim\mathfrak{n}_{-1}=30$  \\\hline
	\end{tabular}
\end{center}
\end{table}

\end{prop}

\begin{prop}[The case of $\mathfrak{e}_8$]
There are two non isomorphic $|3|-$gradings $\mathfrak{n}_{-3}\oplus \cdots \oplus \mathfrak{n}_{3}$ of $\mathfrak{e}_8$:
\begin{itemize}
\item If $\Sigma=\{\alpha_2\}$ then
\[
\dim\mathfrak{n}_{-3}=8,\quad\dim\mathfrak{n}_{-2}=28,\quad\dim\mathfrak{n}_{-1}=56.
\]
\item If $\Sigma=\{\alpha_7\}$, then
\[
\dim\mathfrak{n}_{-3}=2,\quad\dim\mathfrak{n}_{-2}=27,\quad\dim\mathfrak{n}_{-1}=54.
\]
\end{itemize}

\end{prop}

\section{Structure of $\mathfrak{n}_{0}$ for $|3|-$gradings of simple Lie algebras}\label{n0}


In this section we will describe the algebraic structure of $\mathfrak{n}_0$ for all the $|3|-$gradings of the simple Lie algebras. It is a key observation that $\mathfrak{n}_0$ is reductive, see \cite[Theorem 3.2.1]{cap}, and thus it is the direct sum of its center and a semisimple Lie algebra.

The following proposition is a fundamental tool for this purpose.

\begin{prop}[$\mbox{\cite[Proposition 3.2.2]{cap}}$]\label{th:nodes}
Let $\mathfrak{n}=\mathfrak{n}_{-k}\oplus \cdots \oplus \mathfrak{n}_k$ be a complex $|k|-$graded Lie  algebra. The dimension of the center of $\mathfrak{n}_0$ coincides with the number of elements in $\Sigma$, and the Dynkin diagram of its semisimple part is obtained by removing all nodes corresponding to elements of $\Sigma$ and all edges connected to these nodes.
\end{prop}

Using Theorem \ref{th:capslovak} and  Proposition \ref{th:nodes}  we obtain the following description of $\mathfrak{n}_0$ for a $|3|-$grading associated to subset $\Sigma$ of simple roots of simple Lie algebra.

\newpage

\begin{teo}\label{th:n0An}
Let $\mathfrak{n}_{-3}\oplus \cdots \oplus\mathfrak{n}_3$ a $|3|-$grading of a simple Lie algebra over $\mathbb{C}$ of type $A_n$. Then the possible reductive subalgebras $\mathfrak{n}_0$ are contained in Table \ref{n0An}.

\begin{table}[h]
\caption{$\mathfrak{n}_0$ for Lie algebras of type $A_n$}\label{n0An}
\begin{tabular}{| c | c | c | }
\hline
Case & $\Sigma$ & $\mathfrak{n}_0$ \\ \hline
\multirow{2}{*}{1}& $\{\alpha_{1},\alpha_2,\alpha_3\}$ &  \multirow{2}{*}{$\mathbb{C}^3\oplus A_{n-3}$} \\ 
&$\{\alpha_{1},\alpha_2,\alpha_n\}$&\\\hline
\multirow{2}{*}{2}& $\{\alpha_i,\alpha_{i+1},\alpha_{i+2}\}$ & \multirow{2}{*}{$\mathbb{C}^3\oplus A_{i-1}\oplus A_{n-i-2}$} \\ 
& $2\leq i<n-3 $ &   \\ \hline
\multirow{4}{*}{3}& $\{\alpha_i,\alpha_{i+1},\alpha_k\}$ & \multirow{4}{*}{$\mathbb{C}^3\oplus A_{i-1}\oplus A_{k-i-2}\oplus A_{n-k}$} \\ 
& $2\leq i\leq n-3 $ &   \\
& $|k-(i+1)|>1$ &  \\
& $i<k\leq n-1$  & \\ \hline
\multirow{2}{*}{4}&$\{\alpha_i,\alpha_{i+1},\alpha_n\}$ & \multirow{2}{*}{$\mathbb{C}^3\oplus A_{i-1}\oplus A_{n-i-2}$}\\ 
&$1<i<n-1$&\\\hline
\multirow{3}{*}{5}&$\{\alpha_i,\alpha_j,\alpha_k\}$ & \multirow{3}{*}{$\mathbb{C}^3\oplus A_{i-1}\oplus A_{j-i-1}\oplus A_{k-j-1}\oplus A_{n-k}$} \\ 
&$j-i>1$, $k-j>1$ & \\
&$k\leq n-1$  & \\ \hline
\multirow{2}{*}{6}&$\{\alpha_i,\alpha_j,\alpha_n\}$ & \multirow{2}{*}{$\mathbb{C}^3\oplus A_{i-1}\oplus A_{j-i-1}\oplus A_{n-j-1}$}\\ 
&$1<i<j<n-1$&\\ \hline

\end{tabular}

\end{table}
\end{teo}

\begin{proof}
Before proceeding using a case-by-case analysis, we observe that since all sets $\Sigma$ for $A_n$ have three elements, then the center of $\mathfrak{n}_0$ in all cases is simply ${\mathbb C}^3$. Also it is important to remark that duplications may appear in the arguments in trivially isomorphic situations, that is, those that correspond to automorphisms of the Dynkin diagrams. This choice has been taken to make the classification simpler. In the case where it is easy to avoid them, we will do it.
\begin{description}
\item[Case 1] In this situation, we have the following Dynkin diagrams

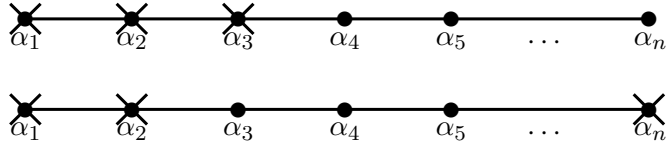
\begin{center}
\begin{figure}[h]
\psscalebox{1.0 1.0} 
{
\begin{pspicture}(0,-4.8)(8.684142,-3.663787)
\psline[linecolor=black, linewidth=0.04](0.22414215,-5.087929)(8.424142,-5.087929)
\psdots[linecolor=black, dotsize=0.2000061](0.22414215,-5.087929)
\psdots[linecolor=black, dotsize=0.2000061](1.6241422,-5.087929)
\psdots[linecolor=black, dotsize=0.2000061](3.0241423,-5.087929)
\psdots[linecolor=black, dotsize=0.2000061](4.4241424,-5.087929)
\psdots[linecolor=black, dotsize=0.2000061](5.824142,-5.087929)
\psdots[linecolor=black, dotsize=0.2000061](8.424142,-5.087929)
\rput[bl](6.824142,-5.487929){$\cdots$}
\rput[bl](0.024142152,-5.487929){$\alpha_1$}
\rput[bl](2.8241422,-5.487929){$\alpha_3$}
\rput[bl](1.4241421,-5.487929){$\alpha_2$}
\rput[bl](4.224142,-5.487929){$\alpha_4$}
\rput[bl](5.624142,-5.487929){$\alpha_5$}
\rput[bl](8.224142,-5.487929){$\alpha_n$}
\psline[linecolor=black, linewidth=0.04](0.22414215,-3.887929)(8.424142,-3.887929)
\psdots[linecolor=black, dotsize=0.2000061](0.22414215,-3.887929)
\psdots[linecolor=black, dotsize=0.2000061](1.6241422,-3.887929)
\psdots[linecolor=black, dotsize=0.2000061](3.0241423,-3.887929)
\psdots[linecolor=black, dotsize=0.2000061](4.4241424,-3.887929)
\psdots[linecolor=black, dotsize=0.2000061](5.824142,-3.887929)
\psdots[linecolor=black, dotsize=0.2000061](8.424142,-3.887929)
\rput[bl](6.824142,-4.287929){$\cdots$}
\rput[bl](0.024142152,-4.287929){$\alpha_1$}
\rput[bl](2.8241422,-4.287929){$\alpha_3$}
\rput[bl](1.4241421,-4.287929){$\alpha_2$}
\rput[bl](4.224142,-4.287929){$\alpha_4$}
\rput[bl](5.624142,-4.287929){$\alpha_5$}
\rput[bl](8.224142,-4.287929){$\alpha_n$}
\psline[linecolor=black, linewidth=0.04](0.42414215,-4.087929)(0.024142152,-3.687929)
\psline[linecolor=black, linewidth=0.04](0.024142152,-4.087929)(0.42414215,-3.687929)
\psline[linecolor=black, linewidth=0.04](1.4241421,-3.687929)(1.8241421,-4.087929)
\psline[linecolor=black, linewidth=0.04](1.4241421,-4.087929)(1.8241421,-3.687929)
\psline[linecolor=black, linewidth=0.04](2.8241422,-3.687929)(3.224142,-4.087929)
\psline[linecolor=black, linewidth=0.04](2.8241422,-4.087929)(3.224142,-3.687929)
\psline[linecolor=black, linewidth=0.04](8.224142,-4.887929)(8.624142,-5.287929)
\psline[linecolor=black, linewidth=0.04](8.224142,-5.287929)(8.624142,-4.887929)
\psline[linecolor=black, linewidth=0.04](0.024142152,-4.887929)(0.42414215,-5.287929)
\psline[linecolor=black, linewidth=0.04](0.024142152,-5.287929)(0.42414215,-4.887929)
\psline[linecolor=black, linewidth=0.04](1.4241421,-4.887929)(1.8241421,-5.287929)
\psline[linecolor=black, linewidth=0.04](1.4241421,-5.287929)(1.8241421,-4.887929)
\end{pspicture}
}
\end{figure}
\end{center}
where we have crossed the elements of each set of roots $\Sigma=\{\alpha_1,\alpha_2,\alpha_3\}$ or $\Sigma=\{\alpha_1,\alpha_2,\alpha_n\}$, respectively. After removing the three crossed roots and the corresponding edges, we end up with $\mathfrak{n}_0={\mathbb C}^3\oplus A_{n-3}$ in both cases.

\item[Case 2] In this situation we are crossing three consecutive nodes of the Dynkin diagram not including the first or last node. Removing them separates the Dynkin diagram into two connected components. For example, if $n\geq5$ and $\Sigma=\{\alpha_2,\alpha_{3},\alpha_{4}\}$, we have the following situation:

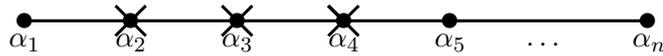
\begin{center}
\begin{figure}[h]
\psscalebox{1.0 1.0} 
{
\begin{pspicture}(0,-1.4)(8.66,-1.1487868)
\psline[linecolor=black, linewidth=0.04](0.2,-3.572929)(0.2,-3.572929)
\psline[linecolor=black, linewidth=0.04](0.2,-1.372929)(8.4,-1.372929)
\psdots[linecolor=black, dotsize=0.2000061](0.2,-1.372929)
\psdots[linecolor=black, dotsize=0.2000061](1.6,-1.372929)
\psdots[linecolor=black, dotsize=0.2000061](3.0,-1.372929)
\psdots[linecolor=black, dotsize=0.2000061](4.4,-1.372929)
\psdots[linecolor=black, dotsize=0.2000061](5.8,-1.372929)
\psdots[linecolor=black, dotsize=0.2000061](8.4,-1.372929)
\rput[bl](6.8,-1.772929){$\cdots$}
\rput[bl](0.0,-1.772929){$\alpha_1$}
\rput[bl](2.8,-1.772929){$\alpha_3$}
\rput[bl](1.4,-1.772929){$\alpha_2$}
\rput[bl](4.2,-1.772929){$\alpha_4$}
\rput[bl](5.6,-1.772929){$\alpha_5$}
\rput[bl](8.2,-1.772929){$\alpha_n$}
\psline[linecolor=black, linewidth=0.04](1.4,-1.1729289)(1.8,-1.5729289)
\psline[linecolor=black, linewidth=0.04](1.4,-1.5729289)(1.8,-1.1729289)
\psline[linecolor=black, linewidth=0.04](2.8,-1.1729289)(3.2,-1.5729289)
\psline[linecolor=black, linewidth=0.04](2.8,-1.5729289)(3.2,-1.1729289)
\psline[linecolor=black, linewidth=0.04](4.6,-1.5729289)(4.2,-1.1729289)(4.4,-1.372929)
\psline[linecolor=black, linewidth=0.04](4.2,-1.5729289)(4.6,-1.1729289)
\psline[linecolor=black, linewidth=0.04](0.2,-3.3729289)(0.2,-3.3729289)
\psline[linecolor=black, linewidth=0.04](0.2,-3.572929)(0.2,-3.572929)
\end{pspicture}
}
\end{figure}
\end{center}
and we conclude that $\mathfrak{n}_0={\mathbb C}^3\oplus A_1\oplus A_{n-4}$. In the general case, for the set of roots $\Sigma=\{\alpha_i,\alpha_{i+1},\alpha_{i+2}\}$, $2\leq i<n-3$, we have that $\mathfrak{n}_0=\mathbb{C}^3\oplus A_{i-1}\oplus A_{n-i-2}$.

\item[Case 3] In this situation, only two of the three nodes of the Dynkin diagram we are crossing are consecutive, not including the first or last node. Removing them separates the Dynkin diagram into three connected components. For example, if $n\geq6$ and $\Sigma=\{\alpha_2,\alpha_{3},\alpha_{5}\}$, we have the following situation:

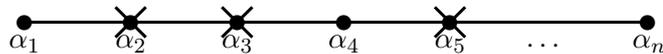
\begin{center}
\begin{figure}[h]
\psscalebox{1.0 1.0} 
{
\begin{pspicture}(0,-1.4)(8.66,-1.1487868)
\psline[linecolor=black, linewidth=0.04](0.2,-3.572929)(0.2,-3.572929)
\psline[linecolor=black, linewidth=0.04](0.2,-1.372929)(8.4,-1.372929)
\psdots[linecolor=black, dotsize=0.2000061](0.2,-1.372929)
\psdots[linecolor=black, dotsize=0.2000061](1.6,-1.372929)
\psdots[linecolor=black, dotsize=0.2000061](3.0,-1.372929)
\psdots[linecolor=black, dotsize=0.2000061](4.4,-1.372929)
\psdots[linecolor=black, dotsize=0.2000061](5.8,-1.372929)
\psdots[linecolor=black, dotsize=0.2000061](8.4,-1.372929)
\rput[bl](6.8,-1.772929){$\cdots$}
\rput[bl](0.0,-1.772929){$\alpha_1$}
\rput[bl](2.8,-1.772929){$\alpha_3$}
\rput[bl](1.4,-1.772929){$\alpha_2$}
\rput[bl](4.2,-1.772929){$\alpha_4$}
\rput[bl](5.6,-1.772929){$\alpha_5$}
\rput[bl](8.2,-1.772929){$\alpha_n$}
\psline[linecolor=black, linewidth=0.04](1.4,-1.1729289)(1.8,-1.5729289)
\psline[linecolor=black, linewidth=0.04](1.4,-1.5729289)(1.8,-1.1729289)
\psline[linecolor=black, linewidth=0.04](2.8,-1.1729289)(3.2,-1.5729289)
\psline[linecolor=black, linewidth=0.04](2.8,-1.5729289)(3.2,-1.1729289)
\psline[linecolor=black, linewidth=0.04](0.2,-3.3729289)(0.2,-3.3729289)
\psline[linecolor=black, linewidth=0.04](0.2,-3.572929)(0.2,-3.572929)
\psline[linecolor=black, linewidth=0.04](5.6,-1.1729289)(6.0,-1.5729289)(6.0,-1.5729289)
\psline[linecolor=black, linewidth=0.04](5.6,-1.5729289)(6.0,-1.1729289)
\end{pspicture}
}
\end{figure}
\end{center}
and we conclude that $\mathfrak{n}_0={\mathbb C}^3\oplus A_1\oplus A_1\oplus A_{n-5}$. In the general case, for the set of roots $\Sigma=\{\alpha_i,\alpha_{i+1},\alpha_k\}$, $2\leq i\leq n-3$, $|k-(i+1)|>1$, $i<k\leq n-1$, we have that $\mathfrak{n}_0=\mathbb{C}^3\oplus A_{i-1}\oplus A_{k-i-2}\oplus A_{n-k}$.

\item[Case 4] In this situation, only two of the three nodes of the Dynkin diagram we are crossing are consecutive, and the remaining node crossed is the last one. Removing them separates the Dynkin diagram into two connected components. For example, if $n\geq5$ and $\Sigma=\{\alpha_2,\alpha_{3},\alpha_{n}\}$, we have the following situation:

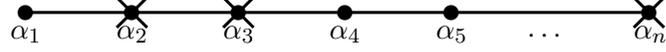
\begin{center}
\begin{figure}[h]
\psscalebox{1.0 1.0} 
{
\begin{pspicture}(0,-1.4)(8.66,-1.1487868)
\psline[linecolor=black, linewidth=0.04](0.2,-3.572929)(0.2,-3.572929)
\psline[linecolor=black, linewidth=0.04](0.2,-1.372929)(8.4,-1.372929)
\psdots[linecolor=black, dotsize=0.2000061](0.2,-1.372929)
\psdots[linecolor=black, dotsize=0.2000061](1.6,-1.372929)
\psdots[linecolor=black, dotsize=0.2000061](3.0,-1.372929)
\psdots[linecolor=black, dotsize=0.2000061](4.4,-1.372929)
\psdots[linecolor=black, dotsize=0.2000061](5.8,-1.372929)
\psdots[linecolor=black, dotsize=0.2000061](8.4,-1.372929)
\rput[bl](6.8,-1.772929){$\cdots$}
\rput[bl](0.0,-1.772929){$\alpha_1$}
\rput[bl](2.8,-1.772929){$\alpha_3$}
\rput[bl](1.4,-1.772929){$\alpha_2$}
\rput[bl](4.2,-1.772929){$\alpha_4$}
\rput[bl](5.6,-1.772929){$\alpha_5$}
\rput[bl](8.2,-1.772929){$\alpha_n$}
\psline[linecolor=black, linewidth=0.04](1.4,-1.1729289)(1.8,-1.5729289)
\psline[linecolor=black, linewidth=0.04](1.4,-1.5729289)(1.8,-1.1729289)
\psline[linecolor=black, linewidth=0.04](2.8,-1.1729289)(3.2,-1.5729289)
\psline[linecolor=black, linewidth=0.04](2.8,-1.5729289)(3.2,-1.1729289)
\psline[linecolor=black, linewidth=0.04](0.2,-3.3729289)(0.2,-3.3729289)
\psline[linecolor=black, linewidth=0.04](0.2,-3.572929)(0.2,-3.572929)
\psline[linecolor=black, linewidth=0.04](8.2,-1.1729289)(8.6,-1.5729289)(8.6,-1.5729289)
\psline[linecolor=black, linewidth=0.04](8.2,-1.5729289)(8.6,-1.1729289)
\end{pspicture}
}
\end{figure}
\end{center}
and we conclude that $\mathfrak{n}_0={\mathbb C}^3\oplus A_1\oplus A_{n-4}$. In the general case, for the set of roots $\Sigma=\{\alpha_i,\alpha_{i+1},\alpha_n\}$, we have that $\mathfrak{n}_0=\mathbb{C}^3\oplus A_{i-1}\oplus A_{n-i-2}$.

\item[Case 5] In this situation, none of the nodes of the Dynkin diagram we are crossing are consecutive, and we do not cross the first nor the last one. Removing them separates the Dynkin diagram into two connected components. For example, if $n\geq8$ and $\Sigma=\{\alpha_3,\alpha_5,\alpha_{n-1}\}$, we have the following situation:

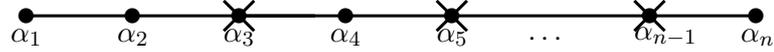
\begin{center}
\begin{figure}[h]
\psscalebox{1.0 1.0} 
{
\begin{pspicture}(0,-0.4)(10.06,-0.042929076)
\psline[linecolor=black, linewidth=0.04](0.2,-2.472929)(0.2,-2.472929)
\psline[linecolor=black, linewidth=0.04](0.2,-0.27292907)(9.8,-0.27292907)
\psdots[linecolor=black, dotsize=0.2000061](0.2,-0.27292907)
\psdots[linecolor=black, dotsize=0.2000061](1.6,-0.27292907)
\psdots[linecolor=black, dotsize=0.2000061](3.0,-0.27292907)
\psdots[linecolor=black, dotsize=0.2000061](4.4,-0.27292907)
\psdots[linecolor=black, dotsize=0.2000061](5.8,-0.27292907)
\psdots[linecolor=black, dotsize=0.2000061](8.4,-0.27292907)
\rput[bl](6.8,-0.67292905){$\cdots$}
\rput[bl](0.0,-0.67292905){$\alpha_1$}
\rput[bl](2.8,-0.67292905){$\alpha_3$}
\rput[bl](1.4,-0.67292905){$\alpha_2$}
\rput[bl](4.2,-0.67292905){$\alpha_4$}
\rput[bl](5.6,-0.67292905){$\alpha_5$}
\rput[bl](9.6,-0.67292905){$\alpha_n$}
\psline[linecolor=black, linewidth=0.04](2.8,-0.07292908)(3.2,-0.4729291)
\psline[linecolor=black, linewidth=0.04](2.8,-0.4729291)(3.2,-0.07292908)
\psline[linecolor=black, linewidth=0.04](0.2,-2.2729292)(0.2,-2.2729292)
\psline[linecolor=black, linewidth=0.04](0.2,-2.472929)(0.2,-2.472929)
\psline[linecolor=black, linewidth=0.04](8.2,-0.07292908)(8.6,-0.4729291)(8.6,-0.4729291)
\psline[linecolor=black, linewidth=0.04](8.2,-0.4729291)(8.6,-0.07292908)
\psline[linecolor=black, linewidth=0.04](3.0,-0.27292907)(4.0,-0.27292907)
\psline[linecolor=black, linewidth=0.04](5.6,-0.4729291)(6.0,-0.07292908)(6.0,-0.07292908)
\psline[linecolor=black, linewidth=0.04](5.6,-0.07292908)(6.0,-0.4729291)(6.0,-0.4729291)
\psline[linecolor=black, linewidth=0.04](6.6,-4.472929)(6.4,-4.472929)(6.6,-4.6729293)
\rput[bl](8.2,-0.67292905){$\alpha_{n-1}$}
\psdots[linecolor=black, dotsize=0.2000061](9.8,-0.27292907)
\end{pspicture}
}
\end{figure}
\end{center}
and we conclude that $\mathfrak{n}_0=\mathbb{C}^3\oplus A_2\oplus A_1\oplus A_{n-7}\oplus A_1$. In the general case, for the set of roots $\Sigma=\{\alpha_i,\alpha_j,\alpha_k\}$, $j-i>1$, $k-j>1$, $k\leq n-1$, we have that $\mathbb{C}^3\oplus A_{i-1}\oplus A_{j-i-1}\oplus A_{k-j-1}\oplus A_{n-k}$.

\item[Case 6] In this situation, none of the nodes of the Dynkin diagram we are crossing are consecutive, and we cross the last one. Removing them separates the Dynkin diagram into three connected components. For example, if $n\geq7$ and $\Sigma=\{\alpha_3,\alpha_5,\alpha_{n}\}$, we have the following situation:

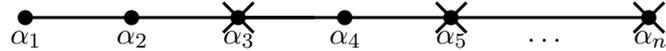
\begin{center}
\begin{figure}[h]
\psscalebox{1.0 1.0} 
{
\begin{pspicture}(0,-1.4)(8.66,-1.14)
\psline[linecolor=black, linewidth=0.04](0.2,-3.57)(0.2,-3.57)
\psline[linecolor=black, linewidth=0.04](0.2,-1.37)(8.4,-1.37)
\psdots[linecolor=black, dotsize=0.2000061](0.2,-1.37)
\psdots[linecolor=black, dotsize=0.2000061](1.6,-1.37)
\psdots[linecolor=black, dotsize=0.2000061](3.0,-1.37)
\psdots[linecolor=black, dotsize=0.2000061](4.4,-1.37)
\psdots[linecolor=black, dotsize=0.2000061](5.8,-1.37)
\psdots[linecolor=black, dotsize=0.2000061](8.4,-1.37)
\rput[bl](6.8,-1.77){$\cdots$}
\rput[bl](0.0,-1.77){$\alpha_1$}
\rput[bl](2.8,-1.77){$\alpha_3$}
\rput[bl](1.4,-1.77){$\alpha_2$}
\rput[bl](4.2,-1.77){$\alpha_4$}
\rput[bl](5.6,-1.77){$\alpha_5$}
\rput[bl](8.2,-1.77){$\alpha_n$}
\psline[linecolor=black, linewidth=0.04](2.8,-1.17)(3.2,-1.57)
\psline[linecolor=black, linewidth=0.04](2.8,-1.57)(3.2,-1.17)
\psline[linecolor=black, linewidth=0.04](0.2,-3.37)(0.2,-3.37)
\psline[linecolor=black, linewidth=0.04](0.2,-3.57)(0.2,-3.57)
\psline[linecolor=black, linewidth=0.04](8.2,-1.17)(8.6,-1.57)(8.6,-1.57)
\psline[linecolor=black, linewidth=0.04](8.2,-1.57)(8.6,-1.17)
\psline[linecolor=black, linewidth=0.04](3.0,-1.37)(4.0,-1.37)
\psline[linecolor=black, linewidth=0.04](5.6,-1.57)(6.0,-1.17)(6.0,-1.17)
\psline[linecolor=black, linewidth=0.04](5.6,-1.17)(6.0,-1.57)(6.0,-1.57)
\end{pspicture}
}
\end{figure}
\end{center}
and we conclude that $\mathfrak{n}_0=\mathbb{C}^3\oplus A_2\oplus A_1\oplus A_{n-6}$. In the general case, for the set of roots $\Sigma=\{\alpha_i,\alpha_j,\alpha_n\}$, $1<i<j<n-1$, we have that $\mathbb{C}^3\oplus A_{i-1}\oplus A_{j-i-1}\oplus A_{n-j-1}$.\qedhere

\end{description}
\end{proof}

\begin{teo}\label{th:n0BCD}
Let $\mathfrak{n}_{-3}\oplus \cdots \oplus\mathfrak{n}_3$ a $|3|-$grading of a simple Lie algebra over $\mathbb{C}$ of type $B_n$, $C_n$ or $D_n$. Then the possible reductive subalgebras $\mathfrak{n}_0$ are contained in Table \ref{n0BnCnDn}.

\begin{table}[h]
\caption{$\mathfrak{n}_0$ for Lie algebras of type  $B_n$, $C_n$ and $D_n$}\label{n0BnCnDn}
\begin{tabular}{|c|c|c|}
\hline
Algebra & $\Sigma$ & $\mathfrak{n}_0$ \\ \hline
\multirow{5}{*}{$B_n$} & $\{\alpha_1,\alpha_2\}$  & $\mathbb{C}^2\oplus B_{n-2}$ \\ \cline{2-3}
&   $\{\alpha_1,\alpha_i\}$ & \multirow{2}{*}{$\mathbb{C}^2\oplus A_{i-2}\oplus B_{n-i}$} \\ 
&$3\leq i\leq n-2$ &\\\cline{2-3}
 & $\{\alpha_1,\alpha_{n-1}\}$  & $\mathbb{C}^2\oplus A_{n-3}\oplus A_1$ \\ \cline{2-3}
 & $\{\alpha_1,\alpha_{n}\}$  & $\mathbb{C}^2\oplus A_{n-2}$ \\ \hline
\multirow{4}{*}{$C_n$} & $\{\alpha_1,\alpha_n\}$ & \multirow{2}{*}{$\mathbb{C}^2\oplus A_{n-2}$} \\ 
& $\{\alpha_{n-1},\alpha_{n}\}$&\\ \cline{2-3}
 & $\{\alpha_i,\alpha_{n}\}$ & \multirow{2}{*}{$\mathbb{C}^2\oplus A_{i-1}\oplus A_{n-i-1}$} \\ 
& $2\leq i\leq n-2$  &\\\hline
\multirow{5}{*}{$D_n$} & $\{\alpha_1,\alpha_i\}$ & \multirow{2}{*}{$\mathbb{C}^2\oplus A_{i-2}\oplus D_{n-i}$} \\ 
&$2\leq i\leq n-3$ &\\ \cline{2-3}
 & $\{\alpha_1,\alpha_{n-2}\}$  & $\mathbb{C}^2\oplus A_{n-4}\oplus A_1\oplus A_1$ \\ \cline{2-3}
&  $\{\alpha_i,\alpha_{n-1}\}$  & \multirow{2}{*}{$\mathbb{C}^2\oplus A_{i-1}\oplus A_{n-i-1}$} \\ 
&$2\leq i\leq n-3$&\\\cline{2-3}
&  $\{\alpha_{n-2},\alpha_{n-1}\}$  & $\mathbb{C}^2\oplus A_1\oplus A_{n-3}$ \\ \cline{2-3}
 &   $\{\alpha_1,\alpha_{n-1},\alpha_{n}\}$  & $\mathbb{C}^3\oplus A_{n-3}$ \\ \hline
\end{tabular}
\end{table}
\end{teo}

\begin{proof}
In spirit, this proof is similar to Theorem \ref{th:n0An}, but there are some technical differences that need to be mentioned in each case. The combinatorics to deal with in all of these cases is in fact easier than the cases of $A_n$. We will use the roots from Table \ref{highroots}.

\begin{description}
\item[The $B_n$ case] Recall that the highest weight root for this Lie algebra is
\[
\theta=\alpha_1+2\alpha_2+2\alpha_3+\cdots+2\alpha_n.
\]
It follows immediately that there are only a few choices of possible subsets $\Sigma$, in order for $\theta$ to have height 3 with respect to $\Sigma$. All these possibilities must include the root $\alpha_1$ and then consider any of the other roots. As an immediate consequence, the center always has dimension 2. 

\item[The $C_n$ case] Recall that the highest weight root for this Lie algebra is
\[
\theta=2\alpha_1+2\alpha_2+2\alpha_3+\cdots+\alpha_n.
\]
There are even fewer choices of possible subsets $\Sigma$ compared to the previous case. All these possibilities must include the root $\alpha_n$ and then consider any of the other roots. As an immediate consequence, the center always has dimension 2. 

\item[The $D_n$ case] Recall that the highest weight root for this Lie algebra is
\[
\theta=\alpha_1+2\alpha_2+2\alpha_3+\cdots+2\alpha_{n-2}+\alpha_{n-1}+\alpha_{n}.
\]
The choices of possible subsets $\Sigma$ have to be more carefully considered than before. There is a special case in which we cross all roots of height 1, obtaining a center of dimension 3. In all other cases the center always has dimension 2. \qedhere
\end{description}
\end{proof}

\begin{teo}
Let $\mathfrak{n}_{-3}\oplus \cdots \oplus\mathfrak{n}_3$ a $|3|-$grading of an exceptional Lie algebra over $\mathbb{C}$. Then the possible reductive subalgebras $\mathfrak{n}_0$ are contained in Table \ref{n0e8g2f4e6e7}.

\begin{table}[h]
\caption{$\mathfrak{n}_0$ for the exceptional Lie algebras }\label{n0e8g2f4e6e7}
\begin{tabular}{|c|c|c|}
\hline
Algebra & $\Sigma$ & $\mathfrak{n}_0$ \\ \hline
\multirow{4}{*}{$\mathfrak{e}_6$}&$\{\alpha_1,\alpha_2\}$ & \multirow{2}{*}{$\mathbb{C}^2\oplus A_4$} \\ 
& $\{\alpha_1,\alpha_{3}\}$  & \\ \cline{2-3}
& $\{\alpha_1,\alpha_{5}\}$  & $\mathbb{C}^2\oplus A_{3}\oplus A_1$ \\ \cline{2-3}
& $\{\alpha_4\}$ & $\mathbb{C}\oplus A_2\oplus A_1\oplus A_2$ \\ \hline
\multirow{5}{*}{$\mathfrak{e}_7$}& $\{\alpha_1,\alpha_7\}$ & \multirow{2}{*}{$\mathbb{C}^2\oplus D_5$} \\ 
& $\{\alpha_6,\alpha_{7}\}$  &  \\ \cline{2-3}
& $\{\alpha_2,\alpha_{7}\}$  & $\mathbb{C}^2\oplus A_{5}$ \\ \cline{2-3}
& $\{\alpha_3\}$ & $\mathbb{C}\oplus A_1\oplus A_5$ \\ \cline{2-3}
& $\{\alpha_5\}$ & $\mathbb{C}\oplus A_{4}\oplus A_{2}$ \\ \hline
\multirow{2}{*}{$\mathfrak{e}_8$} & $\{\alpha_2\}$  & $\mathbb{C}\oplus A_{7}$ \\ \cline{2-3}
& $\{\alpha_7\}$  & $\mathbb{C}\oplus \mathfrak{e}_6\oplus A_{1}$ \\ \hline
$\mathfrak{f}_4$&  $\{\alpha_2\}$  & $\mathbb{C}\oplus A_{1}\oplus A_2$ \\ \hline
$\mathfrak{g}_2$&  $\{\alpha_1\}$  & $\mathbb{C}\oplus A_1$ \\ \hline
\end{tabular}
\end{table}

\end{teo}

\begin{proof}
This result follows by performing a similar analysis to the cases of the classical Lie algebras studied before.
\end{proof}

\begin{ob}
An interesting observation is that only one $|3|-$grading among all simple Lie algebras whose $\mathfrak{n}_0$ contains an exceptional factor is the $|3|-$grading of $\mathfrak{e}_8$ corresponding to the subset $\Sigma=\{\alpha_7\}$.
\end{ob}
\newpage

\section{Free nilpotent Lie algebras of step three and $|3|-$gradings}\label{freecl}

The aim of this last section is to discuss whether the negatively graded part of a $|3|-$grading of a simple Lie algebra can have the structure of a free nilpotent Lie algebra $\mathfrak{F}_{r,3}$. As it is well-known, this is the case of the Cartan grading of $\mathfrak{g}_2$, see \cite{c}. A related answer regarding the Tanaka prolongation was given in \cite{ben}, using Hall bases in a very careful way, where it is shown that for free nilpotent Lie algebras of steps higher than two, the only non-trivial Tanaka prolongation is the one by Cartan. In this context, a trivial Tanaka prolongation of a graded nilpotent Lie algebra $\mathfrak{n}=\mathfrak{n}_{-k}\oplus\cdots\oplus\mathfrak{n}_{-1}$ is obtained simply by including the factor $\mathfrak{n}_0={\text{Der}}_0(\mathfrak{n})$ of graded derivations of $\mathfrak{n}$. 

The precise question we will try to answer here is the following: When is it possible to find a graded isomorphism between $\mathfrak{F}_{r,3}$ and $\mathfrak{n}_{-3}\oplus \mathfrak{n}_{-2}\oplus \mathfrak{n}_{-1}$, where $\mathfrak{n}=\mathfrak{n}_{-3}\oplus \cdots\oplus \mathfrak{n}_{3}$ is a $|3|-$grading of a simple Lie algebra?

\begin{defn}
Let ${\mathfrak g}=\bigoplus_{n\in{\mathbb Z}}{\mathfrak g}_n$ and ${\mathfrak m}=\bigoplus_{n\in{\mathbb Z}}{\mathfrak m}_n$ be two graded Lie algebras. A Lie algebra isomorphism $\phi\colon{\mathfrak g}\to{\mathfrak m}$ is called a \emph{graded isomorphism} if $\phi|_{\mathfrak{g}_{n}}$ restricts to an isomorphism between $\mathfrak{g}_{n}$ and $\mathfrak{m}_{n}$ for all $n\in{\mathbb Z}$.
\end{defn}

An important result to have in mind is the following theorem that can be found, for example, in \cite[p. 13]{reu}.
\begin{teo}\label{th:reu}
Let $\mathfrak{F}_{r,k}=\mathfrak{f}_{-k}\oplus\cdots\oplus\mathfrak{f}_{-1}$ be the canonical grading of the free nilpotent Lie algebra $\mathfrak{F}_{r,k}$. Then, we have the Witt formula
\[
\dim\mathfrak{f}_{-n}=\frac1n\sum_{d|n}\mu(d)r^{n/d},
\]
where $\mu$ denotes the M\"obius function.
\end{teo}

As a consequence, for the case we are interested in, we replace $k=3$ and obtain
\[
\dim\mathfrak{f}_{-3}=\frac{r^3-r}3,\quad\dim\mathfrak{f}_{-2}=\frac{r^2-r}2,\quad\dim\mathfrak{f}_{-1}=r.
\]

\subsection{The case $A_{n}$.} Recall that the $|3|-$gradings of a simple Lie algebra of type $A_n$ are in one-to-one correspondence with the subsets of simple roots $\Sigma_{i,j,k}=\{\alpha_i,\alpha_j,\alpha_k\}$, $1\leq i<j<k\leq n$, see Theorem \ref{th:n0An}. The following fact is well known, see, for example, \cite[p.2]{hum}.

\begin{lem}\label{lem:comm}
Denote by $E_{pq}$ the matrix with a $1$ in the position $(p,q)$ and $0's$ in the other entries. Then
$[E_{pq},E_{rs}]=E_{pq}E_{rs}-E_{rs}E_{pq}=\left\{ \begin{array}{lcr}
E_{pq} & \mbox{if} \,\ p\neq s, q=r\\
-E_{qr}& \mbox{if} \,\ p=s, q\neq r\\
E_{pp}-E_{qq} & \mbox{if} \,\ p=s, q=r\\
0 & \mbox{if} \,\ p\neq s, q\neq r
 \end{array}
   \right.$
\end{lem}

In the case of simple Lie algebras of type $A_n$, we can conclude that the answer to the question is always negative. The proof of this case is somewhat different to the other classical Lie algebras.

\begin{teo}\label{th:Annotfree}
Let $\mathfrak{n}=\mathfrak{n}_{-3}\oplus \cdots \oplus \mathfrak{n}_{3}$ be  a simple Lie algebra of type $A_{n}$ with a $|3|-$grading associated to a subset $\Sigma_{i,j,k}=\{\alpha_i,\alpha_j,\alpha_k\}$ of simple roots, with $1\leq i<j<k\leq n$. Then there does not exist a graded isomorphism between  $\mathfrak{F}_{r,3}=\mathfrak{f}_{-3}\oplus\mathfrak{f}_{-2}\oplus\mathfrak{f}_{-1}$ and $\mathfrak{n}_{-3}\oplus \mathfrak{n}_{-2}\oplus \mathfrak{n}_{-1}$.
\end{teo}

\begin{proof}
We will show that there always are elements $x,y\in \mathfrak{n}_{-1}$ such that $[x,y]=0.$ This immediately implies that $\mathfrak{n}_{-3}\oplus \mathfrak{n}_{-2}\oplus \mathfrak{n}_{-1}$ is not isomorphic to a free nilpotent Lie algebra.

Recall that $\mathfrak{n}_{-1}$ is described as follows:
$$\mathfrak{n}_{-1}=\left\{\begin{pmatrix}
0 & 0 & 0 & 0\\
A & 0 & 0 & 0\\
0 & B & 0 & 0\\
0 & 0 & C & 0
\end{pmatrix}: A\in M_{(j-i)\times i}, B\in M_{(k-j)\times (j-i)}, C\in M_{(n+1-k)\times (k-j)}
\right\}.$$

Define $x=E_{pq}$ and $y=E_{rs}$ with $p\in \{i+1,\hdots,j\}$ , $q\in \{1,\hdots,i\}$, $r\in \{k+1,\hdots ,n+1\}$ and $s\in \{j+1,\hdots , k\}$. Due to these choices, it is easy to see that $x,y\in{\mathfrak n}_{-1}$. In addition, it follows that $p\neq s$ and $q\neq r$, thus applying Lemma \ref{lem:comm} we obtain $[x,y]=0.$
\end{proof}

\begin{ob}
A similar technique can be used to prove that given a $|k|-$grading $\mathfrak{n}_{-k}\oplus \cdots \oplus \mathfrak{n}_{k}$ of a simple Lie algebra of type $A_n$, where $k\geq4$, then $\mathfrak{F}_{r,k}$ is not isomorphic to $\mathfrak{n}_{-k}\oplus \cdots \oplus \mathfrak{n}_{-1}$ for all $r\geq 2.$
\end{ob}

\subsection{The case $B_{n}$}
Recall that the $|3|-$gradings of Lie algebras of type $B_n$ are in one-to-one correspondence with the subsets $\Sigma_{i}=\{\alpha_1,\alpha_i\}$ of simple roots, with $2\leq i \leq n$, see Theorem \ref{th:n0BCD}.

\begin{teo} Let $\mathfrak{n}=\mathfrak{n}_{-3}\oplus \cdots \oplus \mathfrak{n}_{3}$ be a $|3|-$grading associated to $\Sigma_{i}=\{\alpha_1,\alpha_i\}$ of a simple Lie algebra of type $B_{n}$, where $2\leq i \leq n$. Then there is no graded isomorphism between $\mathfrak{F}_{r,3}$ and $\mathfrak{n}_{-3}\oplus\mathfrak{n}_{-2} \oplus \mathfrak{n}_{-1}$.
\end{teo}	
\begin{proof}
If $\mathfrak{F}_{r,3}$ is graded isomorphic to $\mathfrak{n}_{-3}\oplus \mathfrak{n}_{-2} \oplus \mathfrak{n}_{-1}$, then by Proposition \ref{prop:dimBn} we have

 \begin{align}
	r&=(i-1)(2n+2-2i),\nonumber\\
	\frac{r(r-1)}{2}&=(2n+1-2i)+\frac{(i-1)(i-2)}{2},\label{eq:BnFr32}\\
	\frac{r^3-r}{3}&=i-1.\label{eq:BnFr33}
\end{align}	

Using \eqref{eq:BnFr33} in \eqref{eq:BnFr32} we obtain
$$\frac{r^2-r}{2}=\frac{1}{2} \left( \frac{r^3-r}{3}-1 \right)\frac{r^3-r}{3}+2n-\frac{2}{3}\left(\frac{r^3-r}{3}+1 \right),$$
this implies that,
\begin{equation*}
-r^6+2r^4+7r^3+8r^2-16r=36n-12.
\end{equation*}
For $r\geq 3$ the function $f(r)=-r^6+2r^4+7r^3+8r^2-16r$ is negative and  for $r=2$ the function take the value $24$ and this form we obtain $n=1$, but $n\geq 2.$ Therefore, such a graded isomorphism cannot exist.
\end{proof}

\subsection{The case $C_n$}
Recall that the $|3|-$gradings of Lie algebras of type $C_n$ are in one-to-one correspondence with the subsets $\Sigma_{i}=\{\alpha_i,\alpha_n\}$ of simple roots, with $1\leq i<n$, see Theorem \ref{th:n0BCD}. Although this situation is similar to the case of the Lie algebras of type $B_n$, we include the proof which is technically more involved.

\begin{teo} 
Let $\mathfrak{n}=\mathfrak{n}_{-3}\oplus \cdots \oplus \mathfrak{n}_{3}$ be a $|3|-$grading associated to $\Sigma_{i}=\{\alpha_1,\alpha_i\}$ of a simple Lie algebra of type $C_{n}$, where $1\leq i<n$. Then there is no graded isomorphism between $\mathfrak{F}_{r,3}$ and $\mathfrak{n}_{-3}\oplus\mathfrak{n}_{-2} \oplus \mathfrak{n}_{-1}$.
\end{teo}

\begin{proof}
If $\mathfrak{F}_{r,3}$ is graded isomorphic to $\mathfrak{n}_{-3}\oplus \mathfrak{n}_{-2} \oplus \mathfrak{n}_{-1}$ then by Proposition \ref{prop:dimCn} we have
	\begin{align}
		r &= i(n-i)+(n-i)+\frac{(n-i)(n-i-1)}{2},\nonumber\\ 
		\frac{r(r-1)}{2} &= (n-i)^2,\nonumber\\ 
		\frac{r^3-r}{3} &= i+\frac{i(i-1)}{2},\label{eq:CnFr33}
	\end{align}
and thus,
\begin{align}\label{eq:Cn}
r^2-5r&=-2(n-i)(2i+1).
\end{align}
For $r<0$ or $r>5$ equation \eqref{eq:Cn} does not have integer solution because the function $f(r)=r^2-5r$ is strictly positive when $r<0$ or $r>5.$\\
Using equation \eqref{eq:CnFr33}, we see that $\frac{r^3-r}{3}$ must be a triangular number. It is easy to verify that for $r=2,3,4$ there is no possible solution for $i$. Finally, the case $r=5$ is discarded because $i\neq n.$
\end{proof}

\subsection{The case $D_n$} 
Recall that the $|3|-$gradings of Lie algebras of type $D_n$ are in one-to-one correspondence with the subsets of simple roots $\Sigma_{1,i}=\{\alpha_1,\alpha_{i}\}$, 
$\Sigma_{i,n}=\{\alpha_i,\alpha_{n}\}$, where $2\leq i\leq n-2$, and
$\Sigma_{1,n-1,n}=\{\alpha_1,\alpha_{n-1},\alpha_n\}$, see Theorem \ref{th:n0BCD}.

\begin{teo} 
Let $\mathfrak{n}=\mathfrak{n}_{-3}\oplus \cdots \oplus \mathfrak{n}_{3}$ be a $|3|-$grading associated to $\Sigma \in \{\Sigma_{i,1}, \Sigma_{i,n}, \Sigma_{1,n-1,n}\}$ of a simple Lie algebra of type $D_{n}$, where $2\leq i\leq n-2$. Then there is no graded isomorphism between $\mathfrak{F}_{r,3}$ and $\mathfrak{n}_{-3}\oplus\mathfrak{n}_{-2} \oplus \mathfrak{n}_{-1}$.
\end{teo}

\begin{proof}
If $\mathfrak{F}_{r,3}$ is graded isomorphic to $\mathfrak{n}_{-3}\oplus \mathfrak{n}_{-2} \oplus \mathfrak{n}_{-1}$ then by Proposition \ref{prop:dimCn} we can perform the following case-by-case analysis
\begin{itemize}
\item If $\Sigma=\Sigma_{i,1}$, then
\begin{align*}
    r&= (2n-2i)+i-1,\\
    \frac{r(r-1)}{2}&= \frac{(i-1)(i-2)}{2}+(2n-2i),\\
    \frac{r^3-r}{3}&=i-1.
\end{align*}
In a similar manner as previously done, we obtain that 
$$r^6-2r^4-15r^3-8r^2+24r=36-36n<0$$
For $r\geq 3$ the function $f(r)=r^6-2r^4-15r^3-8r^2+24r$ is positive and for $r=2$ the equality is only possible if $n=3$, but $n\geq 4.$ Therefore there is no graded isomorphism in this case.
\item If $\Sigma=\Sigma_{i,n}$, then
\begin{align}
    r&= \frac{(n-i)(n-i-1)}{2}+i(n-i),\label{eq:DnFr31}\\
    \frac{r(r-1)}{2}&= i(n-i),\label{eq:DnFr32}\\
    \frac{r^3-r}{3}&=\frac{i(i-1)}{2}.\label{eq:DnFr33}
\end{align}
From equation \eqref{eq:DnFr31} it follows that
$$2r=2i(n-i)+(n-i)(n-1)-i(n-i),$$
and using \eqref{eq:DnFr32} we obtain
$$2r=r(r-1)+(n-1)(n-i)-\frac{r(r-1)}{2},$$
or equivalently 
\begin{equation}\label{eq:rni}
    r^2-5r=-2(n-1)(n-i).
\end{equation}
Note that $r^2-5r$ is positive for $r>5$ or $r<0$, and using \eqref{eq:DnFr33} we can discard the values $r=2,3,4$ as solutions of \eqref{eq:rni}. Finally, the case $r=5$ is discarded because $n\neq 1$ and $i\neq n.$
\item If $\Sigma=\Sigma_{1,n-1,n}$, then
\begin{align}
    r&= 3(n-2),\label{eq:DnFr34}\\
    \frac{r(r-1)}{2}&= 2+\frac{(n-2)(n-3)}{2},\nonumber\\
    \frac{r^3-r}{3}&=n-2.\label{eq:DnFr35}
\end{align}
Using \eqref{eq:DnFr34} and \eqref{eq:DnFr35} we obtain $r^3-2r=0$, that is $r\in \{0,\pm \sqrt{2}\}$ but these values are not possible for $r.$\qedhere
\end{itemize}
\end{proof}

\subsection{The exceptional cases}
In order to address the remaining cases, one may argue in a similar way than before, but being much more careful. The following result is obvious, but it will still be very useful in many of the results of this subsection.

\begin{lem}\label{lem:triv}
If there exists an isomorphism between $\mathfrak{F}_{r,3}=\mathfrak{f}_{-3}\oplus \mathfrak{f}_{-2}\oplus \mathfrak{f}_{-1}$ and $\mathfrak{n}_{-3}\oplus \mathfrak{n}_{-2}\oplus \mathfrak{n}_{-1}$, then
$$\dim \mathfrak{n}=\dim \mathfrak{n}_{0}+2\sum_{k=1}^{3}\dim \mathfrak{f}_{-k}=\dim \mathfrak{n}_{0}+\frac{2}{3}r^3+r^2+\frac{1}{3}r.$$
or equivalently
\begin{equation*}
2r^3+3r^2+r+3(\dim \mathfrak{n}_0-\dim \mathfrak{n})=0.
\end{equation*}
\end{lem}

Remember that $\mathfrak{n}_0$ for $\mathfrak{e}_{6}$, $\mathfrak{e}_7$, $\mathfrak{e}_8$ and $\mathfrak{f}_4$ is described in Table \ref{n0e8g2f4e6e7}. The dimensions of these exceptional Lie algebras are 78, 133, 248 and 54, respectively.

\begin{teo}
Let $\mathfrak{n}_{-3}\oplus \cdots \oplus \mathfrak{n}_{3}$ be a $|3|-$grading of $\mathfrak{n}\in \{\mathfrak{e}_6,\mathfrak{e}_7,\mathfrak{e}_8,\mathfrak{f}_4\}.$ Then,  $\mathfrak{F}_{r,3}$ is not isomorphic to $\mathfrak{n}_{-3}\oplus \mathfrak{n}_{-2}\oplus \mathfrak{n}_{-1}$ for all $r\geq 2$.
\end{teo}
\begin{proof}
We analyze case to case.

\begin{itemize}
\item If $\mathfrak{n}=\mathfrak{e}_6$ then from Table \ref{n0e8g2f4e6e7} we have
$\mathfrak{n}_0\cong \left\{ \begin{array}{l}
\mathbb{C}^2\oplus A_{4}\\
\mathbb{C}^2\oplus A_{3}\oplus A_1\\
\mathbb{C}\oplus A_2\oplus A_1 \oplus A_2
 \end{array}
   \right.$
and thus 
\[
\dim \mathfrak{n}_{0}\in \{20,26\}.
\]
Applying Lemma \ref{lem:triv} we obtain the equations
\[
2r^3+3r^2+r-n=0,
\]
where $n\in \{156,174\}$. These equations do not have integer solutions for $r.$

\item  If$\mathfrak{n}=\mathfrak{e}_{7}$ then from Table \ref{n0e8g2f4e6e7} we have $\mathfrak{n}_0\cong \left\{ \begin{array}{l}
\mathbb{C}^2\oplus D_{5}\\
\mathbb{C}^2\oplus A_{5}\\
\mathbb{C}\oplus A_{1}\oplus A_5\\
\mathbb{C}^2\oplus A_1\oplus A_2
 \end{array}
   \right.$
  and thus 
\[
\dim \mathfrak{n}_{0}\in \{47,37,39,14\}.
\]
Applying Lemma \ref{lem:triv} we obtain the equations
\[
2r^3+3r^2+r-n=0,
\]
where $n\in \{258,288,282,360\}$, but these equations do not have integer solutions for $r.$
   
\item If $\mathfrak{n}=\mathfrak{e}_{8}$ then from Table \ref{n0e8g2f4e6e7} we have $\mathfrak{n}_0\cong \left\{ \begin{array}{l}
\mathbb{C}\oplus A_{7}\\
\mathbb{C}\oplus \mathfrak{e}_{6}\oplus A_1
 \end{array}
   \right.$
   and thus 
\[
\dim \mathfrak{n}_{0}\in \{64,82\},
\]
and applying the Lemma \ref{lem:triv} we obtain the equations
   $$2r^3+3r^2+r-n=0$$ where $n\in \{282,498\}$, but these equations do not have integer solutions for $r.$

\item If $\mathfrak{n}=\mathfrak{f}_4$ then from Table \ref{n0e8g2f4e6e7} we know that $\mathfrak{n}_0\cong \mathbb{C}\oplus A_{1}\oplus A_2$ and thus $\dim \mathfrak{n}_0=12.$ Applying the Lemma \ref{lem:triv} we have $2r^3+3r^2+r-120=0.$ However this equation does not have integer solutions for $r.$\qedhere
   \end{itemize}
   \end{proof}
   
To answer the problem in the case of $\mathfrak{g}_2$ we again use Table \ref{n0e8g2f4e6e7}. From this table, we know that $\mathfrak{n}_0\cong \mathbb{C}\oplus A_{1}$ and thus $\dim \mathfrak{n}_0=4.$

\begin{teo}
Let $\mathfrak{n}_{-3}\oplus \cdots \oplus \mathfrak{n}_3$ be a $|3|-$grading of the exceptional Lie algebra $\mathfrak{g}_{2}.$ If there is a graded isomorphism between $\mathfrak{F}_{r,3}$ and $\mathfrak{n}_{-3}\oplus \mathfrak{n}_{-2}\oplus \mathfrak{n}_{-1}$, then $r=2.$
\end{teo}
\begin{proof}
Suppose there exists a graded isomorphism $\mathfrak{n}_{-3}\oplus \mathfrak{n}_{-2}\oplus \mathfrak{n}_{-1}.$ Applying Lemma \ref{lem:triv} we obtain the equation
$$2r^3+3r^2+r-30=0$$
and the only integer solution of the last equation is $r=2.$
\end{proof}

\begin{ob}
As mentioned before, this possibility corresponds precisely to the $|3|-$grading of $\mathfrak{g}_2$ discovered by Cartan in \cite{c}, also found by Warhurst in \cite{ben}.
\end{ob}

\paragraph{\it Acknowledgments:} The authors would like to thank professor Emilio Lauret from Universidad Nacional del Sur, Argentina, for his suggestions that led to the shorter proof of Theorem \ref{th:Annotfree}.

\end{document}